\documentclass{amsart}
\usepackage{amssymb, amstext, amscd, amsmath,verbatim,lineno}
\usepackage{enumitem}
\usepackage{xcolor}
\input xy
\xyoption{all}

\theoremstyle{plain}
\newtheorem{theorem}{Theorem}[section]
\newtheorem{corollary}[theorem]{Corollary}
\newtheorem{proposition}[theorem]{Proposition}
\newtheorem{lemma}[theorem]{Lemma}
\theoremstyle{definition}

\newtheorem{example}[theorem]{Example}

\newtheorem{definition}[theorem]{Definition}
\newtheorem{notation}[theorem]{Notation}

\newcommand{\bN}{{\mathbb{N}}}
\newcommand{\bbN}{\mathbb{N}}

\newcommand{\bT}{{\mathbb{T}}}

  \newcommand{\E}{{\mathcal{E}}}

\newcommand{\al}{\alpha}

\renewcommand{\phi}{\varphi}
\newcommand{\upchi}{{\raise.35ex\hbox{\ensuremath{\chi}}}}

\newcommand{\spn}{\operatorname{span}}

\newcommand{\ca}{\mathrm{C}^*}

\newcommand{\ol}{\overline}

\newcommand{\susbeteq}{\subseteq} 

\begin{document}
\title{Regular ideals of graph algebras}
\author[J.H. Brown]{Jonathan H. Brown}
\address[J.H. Brown]{
Department of Mathematics\\
University of Dayton\\
300 College Park Dayton\\
OH 45469-2316 U.S.A.} \email{jonathan.henry.brown@gmail.com}

\author[A.H. Fuller]{Adam H. Fuller}
\address[A.H. Fuller]{
Department of Mathematics\\
Ohio University\\
Athens\\
OH 45701 U.S.A.}
\email{fullera@ohio.edu}

\author[D.R. Pitts]{David R. Pitts}\thanks{This work was supported by a grants from the Simons Foundation (DRP \#316952, SAR \#36563); and by the American Institute of Mathematics SQuaREs Program.}  \address[D.R. Pitts]{
  Department of Mathematics\\
  University of Nebraska-Lincoln\\
  Lincoln\\
  NE 68588-0130 U.S.A.}  \email{dpitts2@unl.edu}

\author[S.A. Reznikoff]{Sarah A. Reznikoff}
\address[S.A. Reznikoff]{
Department of Mathematics\\
Kansas State University\\
138 Cardwell Hall\\
Manhattan, KS, U.S.A. }
\email{sarahrez@math.ksu.edu}

\date{\today}

\begin{abstract}
Let $C^*(E)$ be the graph $\ca$-algebra of a row-finite graph $E$.
We give a complete description of the vertex sets of the gauge-invariant regular ideals of $C^*(E)$.
It is shown that when $E$ satisfies Condition~(L) the regular ideals $C^*(E)$ are a class of gauge-invariant ideals which preserve Condition~(L) under quotients.
That is, we show that if $E$ satisfies Condition~(L) then a regular ideal $J \unlhd C^*(E)$ is necessarily gauge-invariant.
Further, if $J \unlhd C^*(E)$ is a regular ideal, it is shown that $C^*(E)/J \simeq C^*(F)$ where $F$ satisfies Condition~(L).
\end{abstract}

\maketitle

\section{Introduction}
In this short note we study the regular ideals, in the sense of Hamana \cite{Ham1982}, of a row-finite graph $\ca$-algebra $C^*(E)$.
That is the ideals $J$ such that $J=J^{\perp\perp}$, see Definition~\ref{def: reg}.

Let $E$ be a row-finite directed graph.
There is a natural one-to-one correspondence between the saturated, hereditary subsets of vertices of $E$ and the gauge-invariant ideals of $C^*(E)$ \cite{BPRS2000}.
In Proposition~\ref{prop: gauge inv reg descrip} we give a complete description of the gauge-invariant regular ideals $J$ of $C^*(E)$ in terms of their corresponding vertex sets $H(J) = \{v \in E^0 \colon p_v \in J\}$.
Using this information we move to study quotients of graph C$^*$-algebras by gauge-invariant regular ideals.

Two important classes of directed graphs are those satisfying
Condition~(K) and those satisfying Condition~(L).  
Condition~(L) was introduced in \cite{KPR1998} to prove the
Cuntz-Krieger Uniqueness Theorem for graph $\ca$-algebras, and serves
as an analogue of Condition~(I) from \cite{CunKri1980}.  Condition~(K)
was introduced in \cite{KPRR1997} to study the ideal structure of
graph algebras, and serves as an analogue of Condition~(II) for
Cuntz-Krieger algebras introduced in \cite{Cuntz1981}.
Our interest in this note is with Condition~(L), but it provides useful context to briefly discuss the relationship between quotients and Condition~(K).

$C^*$-algebras of  graphs  satisfying Condition~(K) are 
well-behaved under quotients by gauge-invariant ideals.  Indeed, when
$E$ satisfies Condition~(K) and $J \unlhd C^*(E)$ is a gauge-invariant
ideal, then there is a graph $F$ satisfying Condition~(K) such that
$C^*(E)/J \simeq C^*(F)$.   However,  in general, this is not true of graphs satisfying
Condition~(L).  If $E$ satisfies Condition~(L) and $J \unlhd C^*(E)$
is gauge-invariant, then $C^*(E)/J$ is again a graph $\ca$-algebra,
but it may be impossible to find a graph $F$ satisfying Condition~(L)
such that $C^*(E)/J\simeq C^*(F)$, \cite[Remark~4.5]{BPRS2000}.

The poor behavior of Condition~(L) under quotients by gauge-invariant
ideals can be mitigated with the imposition of an additional
hypothesis: regularity of the ideal $J$.
We show in Theorem~\ref{thm: regular quotient type L}
that if $E$ satisfies Condition~(L) and $J \unlhd C^*(E)$ is a
regular, gauge-invariant ideal then $C^*(E)/J \simeq C^*(F)$, where
$F$ is a graph satisfying Condition~(L).

As an application of Theorem~\ref{thm: regular
  quotient type L}, Proposition~\ref{prop: regular is gauge
  inv} shows that when the graph $E$ satisfies Condition~(L),  all regular ideals of $C^*(E)$ are gauge-invariant.  Thus, if $E$ satisfies Condition~(L),
 the regular ideals of $C^*(E)$ are a class of gauge-invariant
ideals which preserve Condition~(L) under quotients.  However, the
converse need not hold.  Indeed, Example~\ref{ex1} gives
an example of a graph satisfying Condition~(L) and a non-regular ideal
$J$ in $C^*(E)$ for which $C^*(E)/J\simeq C^*(F)$ where $F$ is a graph
also satisfying 
Condition~(L).

\section{Background: graph algebras and their ideals}
We recall some definitions  relevant to graph $\ca$-algebras.
We will follow the arrow conventions used in Raeburn's monograph \cite{RaeGraphAlg}, and refer the reader to \cite{RaeGraphAlg} for further details on graph $\ca$-algebras.

Let $E = (E^0,E^1,r,s)$ be a directed graph.  We will assume
throughout that $E$ is a \emph{row-finite} graph.  That is, we will
assume $r^{-1}(v)$ is a finite set for each vertex $v\in E^0$.  For
$N\in \bbN$, a \textit{path of length $N$} in $E$ is a
sequence of edges $\{\alpha_i\}_{i=1}^N$ such that
$s(\alpha_i)=r(\alpha_{i+1})$.
For a path $\alpha$, let
$\alpha^0:=\{r(\alpha_i)\}\cup\{s(\alpha)_i\}$ be the set of vertices
in the path $\alpha$.  Let $E^*$ be the collection of all vertices
of $E$ and all finite paths of $E$.  We extend the notion of source to
elements of $E^*$ by setting
$s(\alpha_1, \alpha_2, \ldots \alpha_n)=s(\alpha_n)$ and $s(v)=v$ when
$v \in E^0$; we extend the notion of range to all paths in the
analogous way.  A path
$\alpha=\alpha_1, \alpha_2, \ldots \alpha_n \in E^* \setminus E^0$ is
a {\em cycle} if $r(\alpha)=s(\alpha)$.

Fix a universal Cuntz-Krieger $E$-family, $\{s_e,\ p_v \colon e \in
E^1,\ v \in E^0\}$.  Then the  
 graph $\ca$-algebra $C^*(E)$ is the $\ca$-algebra generated by $\{s_e, p_v\}$.
For $\alpha = \{\alpha_i\}_{i=1}^N\in E^*$, denote by $s_\al$ the partial isometry
$$ s_\al = s_{\alpha_1} \ldots s_{\alpha_n}.$$ 
We have
\[
p_vs_\alpha=\delta_{v,r(\alpha)} s_\alpha \quad\text{ and }\quad s_\alpha p_w=\delta_{s(\alpha),w}s_\alpha;
\]
in particular $p_vs_\alpha\neq 0$ if and only if $v=r(\alpha)$.
The universality of  $\{s_e,\ p_v \colon e \in E^1,\ v \in
E^0\}$ implies that for $z\in \bT$ the map
$s_e\mapsto zs_e$ induces an action $\gamma$ of $\bT$ on C$^*(E)$,
called the {\em gauge action}.

Let $D = \ol{\spn}\{s_\al s_\al^* \colon \al \in E^*\}$. 
Then $D$ is an abelian C$^*$-subalgebra of $C^*(E)$.
We refer to $D$ as the \emph{diagonal of $\ca(E)$}.\footnote{This terminology has become standard even though $D$ is a $C^*$-diagonal in the sense of Kumjian~\cite{Kum1986} only if $E$ has no cycles.}

\begin{definition}
  The graph $E$ satisfies \emph{Condition~(L)} if every cycle has an
  entry: that is, for every cycle
  $\alpha=\{\alpha_i\}_{i=1}^n \in E^*$, there exists an $i$ and an
  edge $e \in E^1 \setminus \{\alpha_i\}$ such that
  $r(e)=r(\alpha_i)$.
\end{definition}

A particularly nice property of a graph algebra is that a large class of its ideals can be `seen' by looking at the underlying graph.
We recall now the requisite definitions.

\begin{definition}
Let $E$ be a row-finite graph and let $H \subseteq E^0$.
The set $H$ is \emph{hereditary} if whenever $\mu\in E^*$ satisfies
$r(\mu)\in H$, then $s(\mu)\in H$.

The set $H$ is \emph{saturated} if  for all $v\in E^0$, 
$ \{s(e) \colon e\in E^1, \ r(e)=v\} \subseteq H$ implies $v\in H$.
\end{definition}

\begin{notation}
Let $J$ be an ideal in a graph algebra $C^*(E)$.
Let
$$ H(J) := \{ v \in E^0 \colon p_v \in J \}. $$
If $H \subseteq E^0$ let $I(H)$ be the ideal in $C^*(E)$ generated by $H$, that is, 
$$ I(H) :=\overline{\spn}\{s^{}_\alpha s_\beta^*: \alpha, \beta\in E^*
\text{ and } s(\alpha)=s(\beta)\in H\}.$$
\end{notation}

The following results give a one-to-one correspondence between saturated hereditary subsets of $E^0$ and gauge-invariant ideals of $C^*(E)$.
Further, a quotient of $C^*(E)$ by a gauge-invariant ideal results in another graph $\ca$-algebra.

\begin{theorem}[c.f. {\cite[Theorem~4.1]{BPRS2000}}]\label{thm: gauge inv ideals}
Let $E$ be a row-finite graph.
If $H$ is a subset of $E^0$ then $H(I(H))=H$ if and only if $H$ is saturated and hereditary.

If $J$ is an ideal in $C^*(E)$ then $I(H(J))=J$ if and only if $J$ is gauge-invariant.
In particular, for any ideal $J$, $I(H(J))$ is the largest gauge-invariant ideal contained in $J$.
That is, 
$$ I(H(J))= \bigcap_{z\in\bT} \gamma_z(J). $$
\end{theorem}

For any row-finite graph $E$ and closed ideal $J\subseteq C^*(E)$,
denote by $E/J$ the largest subgraph of $E$ with no vertex
belonging to 
$H(J)$, that is, 
\[E/J:= (E^0 \backslash H(J), E^1\backslash s^{-1}(H(J)), r, s).\]

\begin{proposition}[c.f.\ {\cite[Theorem~4.1]{BPRS2000}}]\label{prop: quotient graph}
Let $J$ be a gauge-invariant ideal of $C^*(E)$.
Then $C^*(E)/J \simeq C^*(E/J)$.
\end{proposition}

We will use the following proposition to produce quotient graphs without Condition~(L).

\begin{proposition}\label{prop: quotient type L}
Let $J$ be an ideal in $C^*(E)$.
If the graph $E/J$ satisfies Condition~(L), then $J$ is gauge-invariant.
\end{proposition}

\begin{proof}
  Let $F=E/J$.
By Proposition~\ref{prop: quotient graph}, we may identify
$C^*(E)/I(H(J))$ with $C^*(F)$.
By definition, $I(H(J))\subseteq J$. Let $N\unlhd C^*(F)$
be the image of $J$ under the quotient map $C^*(E)\rightarrow
C^*(E)/I(H(I))= C^*(F)$.
Then
$$ C^*(E)/J \simeq C^*(F) / N. $$ 

Consider the quotient map $q: C^*(F) \rightarrow C^*(F) / N$. Note that $H(N) \subseteq F^0$ is empty and thus   $q(p_v) \neq 0$  for all $v \in F^0$. 
Since $F$ satisfies Condition~(L), by the Cuntz-Krieger Uniqueness Theorem the quotient map  is injective.  Thus $N$ is trivial, so $J=I(H(J))$.
\end{proof}

\section{Regular ideals and quotients}
Let $A$ be a C$^*$-algebra.
For a subset $X \subseteq A$ we define $X^\perp$ to be the set
$$ X^\perp = \{a \in A \colon ax= xa = 0 \text{ for all }x\in X\}. $$

\begin{definition}\label{def: reg}
We call an ideal $J\unlhd A$ a \emph{regular ideal} if $J = J^{\perp\perp}$.
\end{definition}

Note that, if $J$ is an ideal in $A$, then so is $J^\perp$.
It is always the case that $J \subseteq (J^\perp)^\perp$ and $J^\perp = J^{\perp\perp\perp}$.

In this section we will study the regular ideals of a graph C$^*$-algebras $C^*(E)$.
The main result is Theorem~\ref{thm: regular quotient type L}, which
shows that if $E$ satisfies Condition~(L) and $J$ is a regular,
gauge-invariant ideal of $C^*(E)$, then $E/J$ also satisfies
Condition~(L) and $C^*(E)/J \simeq C^*(E/J)$.  En route
to this result we give a description in  Proposition~\ref{prop: gauge inv reg descrip} of the vertex set $H(J)$ for a
regular ideal $J$.
As a consequence of Theorem~\ref{thm: regular quotient type L} we also
show that when $E$ satisfies Condition~(L) all regular ideals of
$C^*(E)$ are gauge-invariant.  Thus the regular ideals form a class of
gauge-invariant ideals which preserve Condition~(L) under quotients.

\begin{lemma}\label{lem: perp gauge inv}
Let $J$ be a gauge-invariant ideal of a graph algebra $C^*(E)$.
Then $J^\perp$ is a gauge-invariant regular ideal.
\end{lemma}

\begin{proof}
If $a \in J^\perp$ then for any $z\in \bT$ 
$$ \gamma_z(a) J = \gamma_z( aJ) = \{0\} $$
and
$$ J\gamma_z(a) = \gamma_z( Ja) = \{0\}. $$
Hence $\gamma_z(a) \in J^\perp$.
\end{proof}

\begin{notation}  The following notation will be useful for describing the vertex set
$H(J)$ of a gauge-invariant, regular $J$. 
\begin{enumerate}
\item For $w\in E^0$, put
$$ T(w) = \{s(\alpha) \colon \alpha \in E^*,\ r(\alpha) =w\}.$$
 \item If $I \susbeteq C^*(E)$ an ideal, let $\overline{H}(I)\subseteq
   E^0$ be the set
$$ \ol{H}(I) = \{ r(\alpha): \alpha\in E^* \text{ and }  s(\alpha)\in H(I)\}.$$
\end{enumerate}
\end{notation}

\begin{proposition}\label{prop: gauge inv reg descrip}
Let $E$ be a row-finite directed graph.
Let $J \subseteq C^*(E)$ be a gauge-invariant ideal.
Then
\begin{enumerate}
\item $J^\perp = I(E^0\backslash \ol{H}(J))$;
\item $J^{\perp\perp} = I(\{w \in E^0 \colon T(w) \subseteq \ol{H}(J)\})$;
\item $J$ is regular if and only if $H(J) = \{w \in E^0 \colon T(w) \subseteq \ol{H}(J)\}$.
\end{enumerate}
\end{proposition}

\begin{proof}
By Lemma~\ref{lem: perp gauge inv}, $J^\perp$ is gauge-invariant.
Hence by Theorem~\ref{thm: gauge inv ideals} $J^\perp = I(H(J^\perp))$.
Thus, to prove (i), it suffices to prove $H(J^\perp) = E^0\backslash \ol{H}(J)$.

Observe that if $v\in H(J)$ and $s(\alpha) = v$ for some $\alpha\in E^*$
then $s_\alpha \in J$.
Also if $\al \in E^*$ then $p_w s_\al \neq 0$ if and only if
$r(\al) = w$.
Hence, if $w\in \ol{H}(J)$, then there exists
$s_\alpha\in J$ such that $p_w s_\al \neq 0$.  Thus
$p_w \notin J^\perp$ and therefore $w\notin H(J^\perp)$.  Conversely
if $w\notin \ol{H}(J)$, then $p_w s_\al=0$ for all $\al\in E^*$ with
$s(\al)\in H(J)$.  Since $J$ is gauge invariant, 
$J=\overline{\spn}\{s^{}_\al s_\beta^*: s(\alpha)=s(\beta)\in H(J)\}$,
so  $p_w\in J^\perp$. Hence $w\in H(J^\perp)$.

The other properties follow from (i).
\end{proof}

We show now that quotients by \emph{regular} ideals  preserve Condition~(L).

\begin{theorem}\label{thm: regular quotient type L}
Let $E$ be a row-finite graph satisfying Condition~(L).
Let $J$ be a regular, gauge-invariant ideal in $C^*(E)$.
Then $E/J$ satisfies Condition~(L).
\end{theorem}

\begin{proof}
Let $\lambda$ be a cycle in $E/J$.
Since $E$ has Condition~(L), 
\[\E:= \{ e \in E^1 \colon e \text{ is an entrance for }\lambda\}\neq \emptyset.\]
If $\lambda$ has no entry in $F$ then $s(\E) \subseteq H(J)$.
Hence $\lambda^0 \subseteq \ol{H}(J)$.
Since $J$ is regular, it follows from Proposition~\ref{prop: gauge inv reg descrip}, that $\lambda^0 \subseteq H(J)$.
This contradicts $\lambda$ being a cycle in $E/J$. 
Hence $E/J$ has Condition~(L).
\end{proof}

In general, a regular ideal of a graph $\ca$-algebra need not be gauge-invariant.
Indeed, if $E$ is a graph with a single vertex and a single edge then $C^*(E) = C(\bT)$.
The $\ca$-algebra $C(\bT)$ does not contain any non-trivial gauge-invariant ideals.
However, $C(\bT)$ contains many regular ideals.
We will see now that if $E$ satisfies Condition~(L), then regular ideals of $C^*(E)$ are necessarily gauge-invariant.

\begin{lemma}\label{lem: regular support regular}
If $J$ is a regular ideal, then $I(H(J)) \subseteq J$ is a gauge-invariant regular ideal.
\end{lemma}

\begin{proof}
We have that $I(H(J)) \subseteq J$ and
$$ I(H(J)) \subseteq I(H(J))^{\perp\perp} \subseteq J^{\perp\perp} = J. $$
As $I(H(J))$ is gauge-invariant, $I(H(J))^\perp$ and $I(H(J))^{\perp\perp}$ are gauge-invariant by Lemma~\ref{lem: perp gauge inv}.
By Theorem~\ref{thm: gauge inv ideals}, $I(H(J))$ is the largest gauge-invariant ideal in $J$.
It follows that $I(H(J))= I(H(J))^{\perp\perp}$. 
Hence $I(H(J))$ is regular.
\end{proof}

\begin{proposition}\label{prop: regular is gauge inv}
If $E$ is a row-finite graph satisfying Condition~(L), and $J$ is a regular ideal in $C^*(E)$, then $J$ is gauge-invariant.
\end{proposition}

\begin{proof}
Since $J$ is regular, $I(H(J))$ is regular by Lemma~\ref{lem: regular support regular}.
Thus, by Theorem~\ref{thm: regular quotient type L}, $E/J$ has Condition~(L).
It follows that $J$ is gauge-invariant by Proposition~\ref{prop: quotient type L}.
\end{proof}

Proposition~\ref{prop: regular is gauge inv} and Theorem~\ref{thm: regular quotient type L} together give the following corollary.

\begin{corollary}
Let $E$ be a row-finite graph satisfying Condition~(L).
Let $J$ be a regular ideal in $C^*(E)$.
Then $E/J$ satisfies Condition~(L) and $C^*(E)/J \simeq C^*(E/J)$.
\end{corollary}

We end the paper with an example which shows that not all ideals $J$ with $E/J$ satisfying Condition~(L) are regular.

\begin{example} \label{ex1} Consider the following directed graph $E$\\
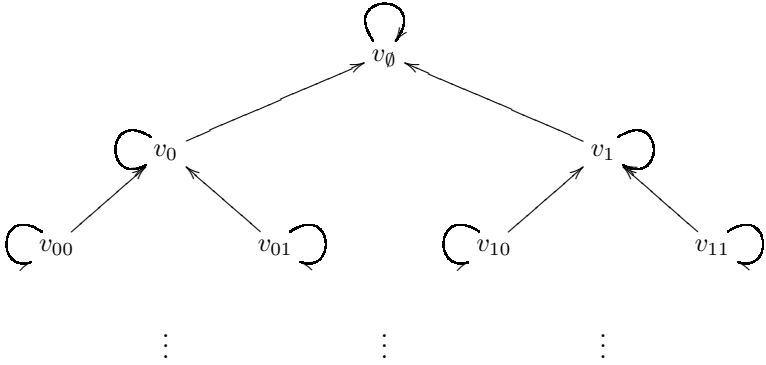
\begin{figure}[h!]
\xymatrix{
&&&v_\emptyset \ar@(ul,ur)&&&\\
& v_{0} \ar@(ul,dl)\ar[urr]&& & &  v_{1}\ar@(ur,dr)\ar[ull]&\\
v_{00}\ar@(ul,dl)\ar[ur]& & v_{01}\ar@(ur,dr)\ar[ul]& & v_{10}\ar@(ul,dl)\ar[ur]& & v_{11}\ar@(ur,dr)\ar[ul] \\
& \vdots & & \vdots & & \vdots &
}
\caption{Graph $E$, with non-regular quotients satisfying Condition~(L)}
\end{figure}

Denote the edge from $v_{j}$ to $v_i$ by $e_{(i,j)}$ and the edge
from $v_i$ to itself by $f_i$.  Since every vertex $v_i$ has an edge $f_i$ with $r(f_i)=s(f_i)=v_i$, every collection of vertices is saturated. 

Now consider the set of vertices $$H=T(v_0)\sqcup\left(\bigsqcup_{i\in \bN-\{0\}} T(v_{{\tiny{\underbrace{11\cdots 1}_{i}}0}})\right).$$  Then $H$ is a saturated hereditary set and $E^0\backslash H=\{v_{{\tiny{\underbrace{11\cdots 1}_{i}}}}:i\in \bN\}\cup \{v_\emptyset\}$.  

Then $E/J$ is represented by Figure~2 and satisfies Condition~(L)\\

\begin{figure}[h!]
\xymatrix{
v_\emptyset \ar@(ul,ur) & v_1\ar@(ul,ur) \ar[l] &  v_{11}\ar@(ul,ur) \ar[l] &\cdots
}
\caption{ $E/J$ with Condition~(L).}
\end{figure}
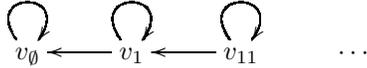

Notice that Proposition~\ref{prop: gauge inv reg descrip} shows $I(H)$
is not regular.  Indeed, the range of
$e_{({\tiny{\underbrace{11\cdots 1}_{i}}}, {\tiny{\underbrace{11\cdots
        1}_{i}}0})}$ is $v_{{\tiny{\underbrace{11\cdots 1}_{i}}}}$ and
its source is in $H$.  So, we have $\overline{H}(I(H))=E^0$, and
therefore \[H(I(H))\neq \{w: T(w)\subseteq \overline{H}(I(H))=E^0\}.\]
\end{example}

\providecommand{\bysame}{\leavevmode\hbox to3em{\hrulefill}\thinspace}
\providecommand{\MR}{\relax\ifhmode\unskip\space\fi MR }
\providecommand{\MRhref}[2]{%
  \href{http://www.ams.org/mathscinet-getitem?mr=#1}{#2}
}
\providecommand{\href}[2]{#2}


\end{document}